\documentclass{amsart}
\usepackage{mathabx}
\usepackage{enumerate}
\newtheorem{Thm}{Theorem}[section]
\newtheorem{Prop}[Thm]{Proposition}
\newtheorem{Cor}[Thm]{Corollary}
\newtheorem{Lem}[Thm]{Lemma}

\theoremstyle{remark}
\newtheorem{Rem}[Thm]{Remark}
\theoremstyle{problem}

\numberwithin{equation}{section}
\begin{document}

\title[Pompeiu problem and spherical spectral analysis]
{The Pompeiu problem and spherical spectral analysis}

\author[M. J. Puls]{Michael J. Puls}
\address{Department of Mathematics \\
John Jay College-CUNY \\
524 West 59th Street \\
New York, NY 10019 \\
USA}
\email{mpuls@jjay.cuny.edu}
\thanks{The author would like to thank the Office for the Advancement of Research at John Jay College for research support for this project}

\begin{abstract}
Let $K$ be a compact subgroup of a locally compact group $G$. We investigate a Pompeiu type problem for homogeneous spaces $G/K$. Suppose $E$ is a compact subset of $G/K$. Using recent work of L\'{a}szl\'{o} Sz\'{e}kelyhidi on $K$-spectral analysis \cite{Szekelyhidi17} we are able to give necessary and sufficient conditions for $E$ to have the Pompeiu property when $(G, K)$ is a Gelfand pair. 
\end{abstract}

\keywords{Gelfand pair, homogeneous space, $K$-biinvariant, $K$-spectral analysis, $K$-translation, Pompeiu property, spherical function}
\subjclass[2020]{Primary: 43A85; Secondary: 22D15, 43A90}

\date{May 5th, 2025}
\maketitle

\section{Introduction}\label{Introduction}
Let $\mathbb{C}$ be the complex numbers, $\mathbb{R}$ the real numbers, $\mathbb{Z}$ the integers and $\mathbb{N}$ the natural numbers. Let $2 \leq n \in \mathbb{N}$ and let $E$ be a compact subset of $\mathbb{R}^n$ with positive Lebesgue measure. The Pompeiu problem asks the following: Is $f = 0$ the only continuous function on $\mathbb{R}^n$ that satisfies 
\begin{equation} 
 \int_{\sigma(E)} f \, dx = 0  \label{eq:contpomp}
\end{equation}
for all rigid motions $\sigma$? If the answer to the question is yes, then $E$ is said to have the Pompeiu property. It is known that disks of positive radius do not have the Pompeiu property, see  \cite[Section 6]{Zalcman80} and the references therein for the details. The question now becomes: Are disks the only compact subsets of positive measure in ${\mathbb{R}^n}$ that do not have the Pompeiu property? This question is still open. It is known though that polytopes in ${\mathbb{R}^n}, n \geq 2$, have the Pompeiu property. This was shown in \cite{Williams76}, a new and simpler proof of this result was recently given in \cite[Corollary 1.3]{MachadoRobins21}.  Various versions of the Pompeiu problem have been studied, see  \cite{BerensteinZalcman80, CareyKaniuthMoran91, LinnellPuls22, PeyerimhoffSamior10, Puls13, RawatSitaram97, ScottSitaram88, Zalcman80} and the references therein for more information about these variations. Pompeiu type problems for finite subsets of the plane were examined in \cite{KissLaczKovichVincze18, Kiss24}. An interesting connection between the Fuglede conjecture and the Pompeiu problem for finite abelian groups was established in \cite{KissMalSomlaiVizer20}.

In this paper we combine ideas from \cite{BerensteinZalcman80} and \cite{Szekelyhidi17} to study a version of the Pompeiu problem for homogeneous spaces. Throughout this paper $G$ will denote a locally compact, second countable, unimodular group with fixed Haar measure $dx$. We denote by $\mathcal{C}(G)$ the set of all continuous complex-valued functions on $G$, equipped with the pointwise operations and with the topology of uniform convergence on compact sets. We shall write $\mathcal{M}_c(G)$ for the space of compactly supported, complex Borel measures on $G$ with the pointwise operations. The convolution of $\mu \in \mathcal{M}_c(G)$ and $f \in \mathcal{C}(G)$ is defined by 
\[ (f \ast \mu) (x) = \int_G f(xy^{-1})\,d\mu(y), \]
where $x \in G$. If $A$ is a measurable subset of a space $X$, then $\chi_A$ will indicate the characteristic function on $A$. Now suppose the measurable subset $A$ of $G$ has finite Haar measure. Then
\[ (f \ast \chi_A )(x) = \int_G f(xy^{-1}) \chi_A (y)\, dy. \]

Let $K$ be a compact subgroup of $G$ with normalized Haar measure $dk$. Denote by $G/K$ the homogeneous space of left cosets $gK$. The measure $d\mu$ on $G/K$ is given by 
\[ \int_G f(x)\, dx = \int_{G/K} (f \ast \chi_K)\, d\mu = \int_{G/K} \int_K f(xk)\,dk d\mu(xK). \]
The group $G$ acts on $G/K$ by left translations $g (xK) = (gx)K$, where $g \in G$. The purpose of this paper is to study the following variation of the Pompeiu problem. Let $E$ be a compact subset of $G/K$. Is $f=0$ the only continuous function on $G/K$ that satisfies 
\[ \int_{gE} f\, d\mu = 0 \]
for all $g \in G$? If the answer to this question is yes, then we shall say that $E$ has the {\em Pompeiu property}. In this paper we will give necessary and sufficient conditions for $E$ to have the Pompeiu property when $(G, K)$ is a Gelfand pair and $G$ satisfies $K$-spectral analysis, which was introduced by Sz\'{e}kelyhidi in \cite{Szekelyhidi17}.

 Let $E$ be a compact subset of $G/K$. In Section \ref{asimplecharacterization} we characterize when $E$ has the Pompeiu property in terms of a convolution equation on $G$. We also define $K$-biinvariant functions and measures. In Section \ref{reduction} we reduce the problem of determining when $E$ has the Pompeiu property to determining when a certain ideal exhausts the algebra of $K$-biinvariant measures with compact support on $G$. We state and prove our main result in Section \ref{mainresultpaper}. A characterization of compact sets without the Pompeiu property in $\mathbb{R}^n$ was given in \cite[Theorem 4.1]{BrownSchreiberTaylor73}. In Section \ref{Euclideangroup} we establish a link between this theorem and our result. 
 
\section{A simple characterization}\label{asimplecharacterization}
In this section we will give a simple characterization for sets to have the Pompeiu property. We begin with some definitions. Let $K$ be a compact subgroup of $G$. We shall say that $f \in \mathcal{C}(G)$ is right (left) $K$-invariant if $f(xk) = f(x)$ $(f(kx) =f(x))$ for all $k \in K$ and $x \in G$. Calculations show that $\chi_K \ast \chi_K = \chi_K$ and $(f \ast \chi_K)(x) = \int_K f(xk)\, dk$, so convolution by $\chi_K$ on the right is a projection from $\mathcal{C}(G)$ to the continuous functions on $G$ that are right $K$-invariant. We identify the continuous functions on $G$ that are right $K$-invariant with the continuous functions on $G/K$ via $f(g) = f(gK)$. Similarly, $\chi_K \ast f$ is the projection from $\mathcal{C}(G)$ to the set of functions on $G$ that are left $K$-invariant. If $E$ is a subset of $G/K$ denote by $\widetilde{E}$ the lift of $E$ to $G$, that is $\widetilde{E} =\{ g \in G \mid gK \in E\}$. For a complex-valued function on $G$, $\widecheck{f}$ will denote the function $\widecheck{f}(x) = f(x^{-1})$. It is easy to see that if $f$ is right $K$-invariant, then $\widecheck{f}$ is left $K$-invariant and vice-versa. We now give a simple characterization of when a compact set in $G/K$ has the Pompeiu property in terms of a convolution equation on $G$.
\begin{Lem}\label{Pompeiuchar}
Let $K$ be a compact subgroup of $G$ and let $E$ be a compact subset of $G/K$ with positive measure. Then $E$ has the Pompeiu property if and only if $f =0$ is the only right $K$-invariant function in $\mathcal{C}(G)$ that satisfies $f \ast \widecheck{\chi}_{\widetilde E} =0$. 
\end{Lem}
\begin{proof}
Let $f$ be a continuous function on $G/K$. Identify $f$ with a continuous function, which we also call $f$, on $G$ that is right $K$-invariant. Observe that for $g \in G$, 
\begin{eqnarray*}
  \int_{gE} f\, d\mu & = &  \int_G f(x)\chi_{g\widetilde{E}}(x)\, dx \\
                             &  =& \int_G f(x) \chi_{\widetilde{E}} (g^{-1}x) \, dx \\
                              & = & \int_G f(x) \widecheck{\chi}_{\widetilde{E}} (x^{-1}g)\, dx \\ 
                               & = &(f \ast \widecheck{\chi}_{\widetilde{E}})(g). 
\end{eqnarray*}
So $\int_{gE} f\, d\mu = 0$ for all $g \in G$ if and only if $f \ast \widecheck{\chi}_{\widetilde{E}} = 0$. Thus $E$ has the Pompeiu  property if and only if $f =0$ is the only right $K$-invariant function in $\mathcal{C}(G)$ that satisfies $f \ast \widecheck{\chi}_{\widetilde{E}} =0$.
\end{proof}

In order to determine if $E$ has the Pompeiu property, the question now becomes: When is $f=0$ the only right $K$-invariant function on $G$ that satisfies $f \ast \widecheck{\chi}_{\widetilde{E}} = 0$?  Before we answer this question we need to discuss $K$-biinvariant functions and measures on $G$. 

Let $K$ be a compact subgroup of $G$. We shall say that $f \in \mathcal{C}(G)$ is $K$-biinvariant if $f(kxl) = f(x)$ for all $k, l \in K$ and $x \in G$. We will denote the space of continuous $K$-biinvariant functions on $G$ by $\mathcal{C}(G//K)$. For $f \in \mathcal{C}(G), \chi_K \ast f \ast \chi_K$ is clearly in $\mathcal{C}(G//K)$ and a calculation shows that 
\[ \chi_K \ast f \ast \chi_K = \int_K \int_K f(lxk)\, dldk. \]
In the spirt of \cite{Szekelyhidi17} we shall write $f^{\#}$ to denote the above projection of $f \in \mathcal{C}(G)$ onto $\mathcal{C}(G//K)$, that is 
\[ f^{\#}(x) = \int_K \int_K f(lxk)\, dl dk. \]
The dual space of $\mathcal{C}(G)$ can be identified with $\mathcal{M}_c(G)$ equipped with the weak*-topology. The continuous linear functional $F\colon \mathcal{C}(G) \rightarrow \mathbb{C}$ is given by 
\[ F(f) = \langle f, \mu \rangle = \int_G f\, d\mu \]
for some $\mu \in \mathcal{M}_c(G)$. For $\mu \in \mathcal{M}_c(G)$, the measure $\widecheck{\mu}$ is defined by 
\[ \langle f, \widecheck{\mu} \rangle = \langle \widecheck{f}, \mu \rangle, \]
where $f \in \mathcal{C}(G)$. The measure $\mu \in \mathcal{M}_c(G)$ is defined to be $K$-biinvariant if 
\[ \langle f, \mu \rangle = \langle f^{\#}, \mu \rangle \]
for all $f \in \mathcal{C}(G)$. Thus the projection $\mu^{\#}$ of $\mu \in \mathcal{M}_c(G)$ is given by
\[ \langle f, \mu^{\#} \rangle = \int_G \int_K \int_K f(lxk)\, dl dk d\mu(x) \]
for all $f \in \mathcal{C}(G)$. Observe, $\langle f, \mu^{\#} \rangle = \langle f^{\#}, \mu \rangle$, so $\mu$ is $K$-biinvariant if and only if $\mu^{\#} = \mu$. Denote by $\mathcal{M}_c (G//K)$ the set of measures in $\mathcal{M}_c(G)$ that are $K$-biinvariant. It was shown in \cite[Theorem 2]{Szekelyhidi17} that $\mathcal{M}_c(G//K)$ is the dual of $\mathcal{C}(G//K)$. 

The convolution of measures $\mu, \nu$ in $\mathcal{M}_c(G//K) $ is given by 
\[ \langle  f , \mu \ast \nu \rangle = \int f(xy)\, d\mu(x) d\nu(y) \]
where $f \in \mathcal{C}(G)$. It was shown in \cite[Sections 1 \& 2]{Szekelyhidi17} that $\mathcal{M}_c(G//K)$ and $\mathcal{M}_c(G)$ are topological algebras with multiplication given by convolution of measures. The point measure $\delta_e$, where $e$ is the identity element in $G$, is the multiplicative unit in both $\mathcal{M}_c(G)$ and $\mathcal{M}_c(G//K)$. 

\section{The Pompeiu property and $\mathcal{M}_c(G//K)$}\label{reduction} 
In this section we prove a result that connects Lemma \ref{Pompeiuchar} with a certain ideal in $\mathcal{M}_c(G//K)$. This will allow us to use results from \cite{Szekelyhidi17} to determine when a compact subset in $G/K$ has the Pompeiu property. We start with some definitions.

Let $f \in \mathcal{C}(G)$ and let $g \in G$. Recall that the right translation of $f$ by $g$ is $R_gf(x) = f(xg)$, and the left translation of $f$ by $g$ is $L_gf(x) = f(g^{-1}x)$, where $x \in G$. However, if $f \in \mathcal{C}(G//K)$ it might be the case that $R_gf$ is not in $\mathcal{C}(G//K)$. We can remedy this situation by taking the right projection $R_gf \ast \chi_K$ of $R_g f$. Specifically for $x \in G$,
\[ Rg f \ast \chi_K (x) = \int_K f(xkg)\, dk. \]
We shall say that $\int_K f(xkg)\, dk$ is the right {\em $K$-translation} of $f$ by $g$, which we denote by $\tau_gf(x)$. Let $g \in G$ and let $\delta_g$ be the point measure concentrated on $g$. For $f \in \mathcal{C}(G//K)$, 
\[ f \ast \delta^{\#}_{g^{-1}} (x) = \int_G f(xy^{-1})\, d\delta^{\#}_{g^{-1}} (y) = \int_G L_{x^{-1}} f(y^{-1})\, d\delta^{\#}_{g^{-1}} (y) = \]
\[ \langle \widecheck{L_{x^{-1}}f}, \delta^{\#}_{g^{-1}} \rangle = \langle (\widecheck{L_{x^{-1}}f})^{\#}, \delta_{g^{-1}} \rangle = \int_G(\widecheck{L_{x^{-1}}f})^{\#} (y) \, d\delta_{g^{-1}}(y) = \]
\[ \int_G \int_K \int_K L_{x^{-1}}f (k^{-1}y^{-1} l^{-1}) \, dl dk d\delta_{g^{-1}}(y)  = \]
\[ \int_K \int_K f(xk^{-1}gl^{-1}) \, dldk = \int_K f(xkg)\, dk = \tau_g(f)(x). \]
Thus $f \ast \delta_{g^{-1}}^{\#} = \tau_g(f)$ when $f \in \mathcal{C}(G//K)$. 

A subset $X \subseteq \mathcal{C}(G//K)$ is said to be {\em $K$-translation invariant} if $\tau_g(f) \in X$ for all $g \in G$ and $f \in X$. A {\em $K$-variety} is a closed $K$-translation invariant linear subspace of $\mathcal{C}(G//K)$. It was shown in \cite[Theorem 3]{Szekelyhidi17} that the finitely supported measures in $\mathcal{M}_c(G//K)$ form a dense subalgebra in $\mathcal{M}_c(G//K)$. Using this result it can be shown that the right $K$-invariant subspaces in $\mathcal{M}_c(G//K)$ are precisely the right ideals in $\mathcal{M}_c(G//K)$, see the proof of \cite[Theorem 5]{Szekelyhidi17} for the details. For a right ideal $I$ in $\mathcal{M}_c(G//K)$ set 
\[ I^{\perp} = \{ f \in \mathcal{C}(G//K) \mid \langle f, \mu \rangle =0 \mbox{ for all }\mu \in I \}. \]
\begin{Lem} \label{dualideal}
Let $I$ be a right ideal in $\mathcal{M}_c(G//K)$. Then $I^{\perp}$ is a $K$-variety. 
\end{Lem}
\begin{proof}
It is easy to see that $I^{\perp}$ is a closed linear space in $\mathcal{C}(G//K)$. Let $f \in I^{\perp}$ and $g\in G$. We need to show that $\tau_g(f) = f \ast \delta_{g^{-1}}^{\#} \in I^{\perp}$. Since $I$ is a right ideal in $\mathcal{M}_c(G//K), \mu \ast \delta_g^{\#} \in I$ for all $\mu \in I$. Let $\mu \in I$, then 
\[ \langle f, \mu \ast \delta_g^{\#} \rangle = \int_G f(t) \, d(\mu \ast \delta_g^{\#}) (t) =\int_G\int_G f(xy) \, d\mu(x)d\delta_g^{\#}(y) = \] 
 \[ \int_G \left(\int_G L_{x^{-1}}f(y) \, d\delta_g^{\#}(y) \right) d\mu(x) = \int_G \left( \int_G (L_{x^{-1}}f(y))^{\#}\, d\delta_g(y) \right) d\mu(x) =  \]
 \[ \int_G \left( \int_G \int_K \int_K f(xkyl)\, dkdld\delta_g(y) \right) d\mu(x) = \]
 \[ \int_G f \ast \delta_{g^{-1}}^{\#} (x) \, d\mu(x) = \langle f\ast \delta_{g^{-1}}^{\#}, \mu \rangle. \]
Since $\langle f, \mu\ast \delta_g^{\#} \rangle = 0$ we see immediately that $f\ast \delta_{g^{-1}}^{\#} \in I^{\perp}$ for all $g \in G$. 
\end{proof}
We will also need
\begin{Lem} \label{annconvolution}
Let $I$ be a right ideal in $\mathcal{M}_c(G//K)$. If $f \in I^{\perp}$, then $\widecheck{f} \ast \mu =0$ for all $\mu \in I$.
\end{Lem}
\begin{proof}
Let $I$ be a right ideal in $\mathcal{M}_c(G//K), \mu \in I$ and let $f \in I^{\perp}$. Then for $g \in G$, 
\[ 0 = \langle f, \mu \ast \delta_g^{\#} \rangle = (\widecheck{f}\ast (\mu \ast \delta_g^{\#}))(e) = ((\widecheck{f} \ast \mu) \ast \delta_g^{\#}) (e) = \int_G (\widecheck{f} \ast \mu ) (y^{-1})\, d\delta_g^{\#} (y) = \]
\[ \langle (\widecheck{f} \ast \mu), \widecheck{\delta_g}^{\#} \rangle =  \langle (\widecheck{f} \ast \mu)^{\#}, \widecheck{\delta_g} \rangle = \langle \widecheck{f} \ast \mu, \widecheck{\delta_g} \rangle = \langle \widecheck{f} \ast \mu, \delta_{g^{-1}} \rangle = \]
\[ \int_G (\widecheck{f }\ast \mu)(y)\, d\delta_{g^{-1}}(y) = (\widecheck{f} \ast \mu ) (g^{-1}). \]
Thus $\widecheck{f} \ast \mu =0$. 
\end{proof}

Let $H$ be a subgroup of $G$. By a left transversal we mean a set of left coset representatives for $H$ in $G$. Let $T$ be a left transversal of $K$ in $G$. For $g \in T$, the characteristic function $\chi_{\widetilde{gK}}$ is right $K$-invariant. If $E$ is a compact subset of $G/K$, then $\chi_{\widetilde{E}}$ is right $K$-invariant and $\widecheck{\chi}_{\widetilde{E}}$ is left $K$-invariant. Thus $\widecheck{\chi}_{\widetilde{E}} \ast \chi_{\widetilde{gK}} \in \mathcal{M}_c (G//K) $ for each $g \in T$. Now let $I$ be the right ideal in $\mathcal{M}_c(G//K)$ that is generated by $\{ \widecheck{\chi}_{\widetilde{E}} \ast \chi_{\widetilde{gK}} \mid g \in T\}$. We are now ready to give a necessary and sufficient condition on $I$ for $E$ to have the Pompeiu property. 
\begin{Thm} \label{Pompeiureduce}
Let $K$ be a compact subgroup of $G$. Let $E$ be a compact subset of $G/K$ and let $I$ be the right ideal in $\mathcal{M}_c(G//K)$ that is generated by the set $\{ \widecheck{\chi}_{\widetilde{E}}\ast \chi_{\widetilde{gK}} \mid g\in T\}$. Then $E$ has the Pompeiu property if and only if $I = \mathcal{M}_c(G//K)$. 
\end{Thm}
\begin{proof}
Assume $I = \mathcal{M}_c(G//K)$. If $f$ is a right $K$-invariant function in $\mathcal{C}(G)$ that satisfies $f \ast \widecheck{\chi}_{\widetilde{E}} = 0$, then $f \ast \mu =0$ for all $\mu \in I$. Consequently, $f=0$ since $\delta_e \in I$. By Lemma \ref{Pompeiuchar}, $E$ has the Pompeiu property. 

Now suppose $I \neq \mathcal{M}_c(G//K)$, then $I^{\perp} \neq \{0\}$. By Lemma \ref{annconvolution}, $\widecheck{f} \ast \mu = 0$ for all $\mu \in I$ and $f \in I^{\perp}$. In particular, there is a nonzero $f \in I^{\perp}$ that satisfies $\widecheck{f} \ast (\widecheck{\chi}_{\widetilde{E}} \ast \chi_{\widetilde{gK}} ) =0$ for all $g \in T$. Since $f$ is $K$-biinvariant, $\widecheck{f}$ is $K$-biinvariant and $\widecheck{f} \ast \widecheck{\chi}_{\widetilde{E}}$ is left $K$-invariant. Now 
\begin{eqnarray*}
   0 & = & (\widecheck{f} \ast (\widecheck{\chi}_{\widetilde{E}} \ast \chi_{\widetilde{gK}}))(e) \\
       & = & \int_G (\widecheck{f} \ast \widecheck{\chi}_{\widetilde{E}})(y^{-1}) \chi_{\widetilde{gK}}(y)\, dy \\
       &  = & \int_K (\widecheck{f} \ast \widecheck{\chi}_{\widetilde{E}})(k^{-1}g^{-1})\, dk \\
       & = & \int_K (\widecheck{f} \ast \widecheck{\chi}_{\widetilde{E}})(g^{-1}) \, dk \\
       & =& (\widecheck{f} \ast \widecheck{\chi}_{\widetilde{E}})(g^{-1}), 
\end{eqnarray*} 
thus $\widecheck{f} \ast \widecheck{\chi}_{\widetilde{E}} (g^{-1}) =0$ for all $g \in T$. Let $h \in G$, because $gK = hK$ for some $g \in T$, there exists a $k \in K$ for which $hk = g$. So, 
\[ (\widecheck{f} \ast \widecheck{\chi}_{\widetilde{E}})(h^{-1}) = (\widecheck{f} \ast \widecheck{\chi}_{\widetilde{E}})(k^{-1}h^{-1}) = (\widecheck{f} \ast \widecheck{\chi}_{\widetilde{E}})(g^{-1}) =0. \]
Hence, $\widecheck{f} \ast \widecheck{\chi}_{\widetilde{E}} = 0$. Lemma \ref{Pompeiuchar} now tells us that $E$ does not have the Pompeiu property. 
\end{proof}

\section{The main result}\label{mainresultpaper}
In this section we will state and prove our main result. First we need some definitions. Let $K$ be a compact subgroup of $G$ and let $T$ be a left transversal of $K$ in $G$. Denote by $I$ the ideal in $\mathcal{M}_c(G//K)$ generated by $\{\widecheck{ \chi}_{\widetilde{E}} \ast \chi_{\widetilde{gK}} \mid g \in T\}$, where $E$ is a compact subset of $G/K$. By Theorem \ref{Pompeiureduce}, $E$ has the Pompeiu property if and only if $I = \mathcal{M}_c(G//K)$. This leads to the obvious question of when is $I = \mathcal{M}_c(G//K)$? We will answer this question in the case of when $(G,K)$ is a Gelfand pair and $K$-spectral analysis holds on $G$. 

We shall say that $(G, K)$ is a {\em Gelfand pair} if the algebra $\mathcal{M}_c (G//K)$ is commutative. The Gelfand pair assumption will allow us to use the theory of commutative algebras. For more information about Gelfand pairs see \cite[Section 4]{Szekelyhidi17}. For the rest of this paper we will assume that $(G,K)$ is a Gelfand pair. Before we are able to define $K$-spectral analysis, we need to define $K$-spherical functions. A $K$-biinvariant function on $G$ is said to be a $K$-spherical function if it satisfies
\[ \int_K f(xky)\, dk = f(x)f(y) \]
for all $x$ and $y$ in $G$. Note that $K$-spherical functions are nonzero since $f(e) =1$; they are also generalizations of exponential functions on abelian groups. See \cite[Section 6]{Szekelyhidi17} for more details about $K$-spherical functions on $G$. Let $\mathcal{S}$ be the set of $K$-spherical functions in $\mathcal{C}(G//K)$. Let $V$ be a $K$-variety in $\mathcal{C}(G//K)$, {\em $K$-spectral analysis holds on $V$} if every nonzero $K$-subvariety of $V$ contains an $f \in \mathcal{S}$. We shall say that {\em $K$-spectral analysis holds on $G$} if $K$-spectral analysis holds on $\mathcal{C}(G//K)$, that is, $K$-spectral analysis holds on every nonzero $K$-subvariety of $\mathcal{C}(G//K)$. 

Another reason $K$-spherical functions are useful to us is that they allow us to define a continuous algebra homomorphism from $\mathcal{M}_c(G//K)$ into $\mathbb{C}$. Specifically, for $f \in \mathcal{S}$ define $\Phi_f \colon \mathcal{M}_c(G//K) \rightarrow \mathbb{C}$ by 
\[ \Phi_f (\mu) = \langle f, \widecheck{\mu} \rangle \]
where $\mu \in \mathcal{M}_c(G//K)$. It was shown in the proof of \cite[Theorem 12]{Szekelyhidi17} that $\Phi_f$ is a continuous algebra homomorphism, thus the kernel of $\Phi_f$ is a maximal ideal in $\mathcal{M}_c(G//K)$. The zero set of $\mu \in \mathcal{M}_c(G//K)$ is defined to be 
\[ Z(\mu) := \{ f \in \mathcal{S} \mid \Phi_f(\mu) =0 \}, \]
and the zero set of an ideal $I$ in $\mathcal{M}_c(G//K)$ is defined by 
\[ Z(I) := \bigcap_{\mu \in I} Z(\mu). \]
We are now ready to state and prove are main result:
\begin{Thm}\label{mainresult}
Let $K$ be a compact subgroup of $G$. Suppose $(G, K)$ is a Gelfand pair and assume $G$ satisfies $K$-spectral analysis. Let $E$ be a compact subset of $G/K$ and let $T$ be a left transversal for $K$ in $G$. Denote by $I$ the ideal in $\mathcal{M}_c(G//K)$ that is generated by $\{\widecheck{\chi}_{\widetilde{E}} \ast \chi_{\widetilde{gK}} \mid g \in T \}$. Then $E$ has the Pompeiu property on $G/K$ if and only if $Z(I) = \emptyset$. 
\end{Thm}
\begin{proof}
Suppose $E$ has the Pompeiu property, then $I=\mathcal{M}_c(G//K)$ by Theorem \ref{Pompeiureduce}. Consequently, $Z(I) = \emptyset$ since $I^{\perp} = \{0\}$. 

Now assume that $E$ does not have the Pompeiu property. Then by Theorem \ref{Pompeiureduce}, $I \neq \mathcal{M}_c(G//K)$. By Lemma \ref{dualideal}, $I^{\perp}$ is a $K$-variety in $\mathcal{C}(G//K)$. Since $G$ satisfies $K$-spectral analysis there exists a nonzero $f \in \mathcal{S}$ that belongs to $I^{\perp}$, so $\langle f, \mu \rangle =0$ for all $\mu \in I$. Now $\widecheck{f} \in \mathcal{S}$ because 
\[ \int_K \widecheck{f}(xky)\, dk = f(y^{-1}) f(x^{-1}) = \widecheck{f}(x) \widecheck{f}(y). \]
Now $\Phi_{\widecheck{f}}(\mu) = \langle f, \mu \rangle$ for $\mu \in \mathcal{M}_c(G//K)$. Thus $\widecheck{f} \in Z(I)$ and $Z(I) \neq \emptyset$.
\end{proof}

\begin{Cor}\label{cormainresult}
Assume the hypothesis of the theorem. Then $E$ has the Pompeiu property if and only if $f \ast \widecheck{\chi}_{\widetilde{E}} \neq 0$ for all $f \in \mathcal{S}$. 
\end{Cor}
\begin{proof}
Necessity follows from Lemma \ref{Pompeiuchar}.  Conversely, if $E$ does not have the Pompeiu property then there exists $f \in \bigcap_{g \in T} Z(\widecheck{\chi}_{\widetilde{E}} \ast \chi_{\widetilde{gK}})$. Hence
\[ 0 = \Phi_f (\widecheck{\chi}_{\widetilde{E}} \ast \chi_{\widetilde{gK}}) = ( f \ast ( \widecheck{\chi}_{\widetilde{E}} \ast \chi_{\widetilde{gK}}))(e). \]
By the calculation in Theorem \ref{Pompeiureduce} we see that $(f \ast (\widecheck{\chi}_{\widetilde{E}} \ast \chi_{\widetilde{gK}}))(e) = (f \ast \widecheck{\chi}_{\widetilde{E}})(g^{-1}) =0$ for all $g \in T$.
\end{proof}

\begin{Rem}\label{Eradial}
Let $E$ be a compact subset of $G/K$. If it is the case that $\chi_{\widetilde{E}} \in \mathcal{M}_c(G//K)$, then we only need to check $Z(\widecheck{\chi}_{\widetilde{E}})$ to determine if $E$ has the Pompeiu property. This simplification follows from the fact that the right ideal $I$ in Theorem \ref{Pompeiureduce} is generated by $\widecheck{\chi}_{\widetilde{E}}$ because $\widecheck{\chi}_{\widetilde{E}} \ast \delta^{\#}_{g^{-1}} = \widecheck{\chi}_{\widetilde{E}} \ast \chi_{\widetilde{gK}}$ when $\chi_{\widetilde{E}} \in \mathcal{M}_c(G//K)$. 
\end{Rem}

\section{The $\mathbb{R}^n$ case}\label{Euclideangroup}
In this section we discuss (\ref{eq:contpomp}) in the context of homogeneous spaces. Let $n \in \mathbb{N}$ with $n \geq 2$ and let $SO(n)$ denote the special orthogonal group. The special Euclidean group $M(n)$ is the semidirect product $M(n) = SO(n) \ltimes \mathbb{R}^n$ with group law 
\[ (k_1, x_1)(k_2, x_2) = (k_1k_2, k_1x_2 + x_1), \]
where $(k_1, x_1), (k_2, x_2) \in SO(n) \times \mathbb{R}^n$. The identity element of $M(n)$ is $(id, 0)$, where $id$ is the identity in $SO(n)$ and $0$ is the additive identity in $\mathbb{R}^n$. The inverse of $(k,x) \in M(n)$ is $(k,x)^{-1} = (k^{-1}, k^{-1}(-x))$. Also, $SO(n)$ is isomorphic to the compact subgroup $\{ (k,0) \mid k \in SO(n)\}$ in $M(n)$, and $\mathbb{R}^n$ is isomorphic to the normal subgroup $\{ (id, x) \mid x \in \mathbb{R}^n \}$ in $M(n)$. Let $E$ be a compact subset of $\mathbb{R}^n$ with positive Lebesgue measure and identify $\mathbb{R}^n$ with the homogeneous space $M(n)/SO(n)$. We can think of $\mathbb{R}^n$ as a left transversal for $M(n)/SO(n)$. Our results can be used in this setting because it was shown in \cite[Corollary 5]{Szekelyhidi17} that $(M(n), SO(n))$ is a Gelfand pair and in \cite[Theorem 20]{Szekelyhidi17} that $M(n)$ satisfies $SO(n)$-spectral analysis. By Theorem \ref{mainresult} $E$ has the Pompeiu property if and only if $\bigcap_{y \in \mathbb{R}^n} Z(\widecheck{\chi}_{\widetilde{E}} \ast \chi_{\widetilde{ySO(n)}}) = \emptyset.$ 

The continuous functions on $M(n)$ that are right $SO(n)$-invariant can be identified with $\mathcal{C}(\mathbb{R}^n)$, the continuous functions on $\mathbb{R}^n$. Indeed, let $(k,x) \in M(n)$ and let $(k', 0) \in SO(n)$. Then $f((k,x)) = f((k,x)(k',0)) = f((kk', x))$ for all $k' \in SO(n)$. Setting $k' = k^{-1}$ we obtain $f((k,x)) = f((id,x))$, so $f$ depends only on $x$. A function $f \in \mathcal{C}(\mathbb{R}^n)$ is said to be {\em radial} if $f(kx) = f(x)$ for all $k \in SO(n)$. We shall write $\mathcal{C}_{SO(n)}(\mathbb{R}^n)$ to indicate the radial functions in $\mathcal{C}(\mathbb{R}^n)$. A measure $\mu \in \mathcal{M}_c(\mathbb{R}^n)$ is radial if
\[ \int_{\mathbb{R}^n} f(x)\, d\mu(x) = \int_{\mathbb{R}^n} f(kx)\, d\mu(x) \]
holds for all $k \in SO(n)$. We denote the radial measures in $\mathcal{M}_c(\mathbb{R}^n)$ by $\mathcal{M}_{SO(n)} (\mathbb{R}^n)$. Sz\'{e}kelyhidi showed in \cite[Section 5]{Szekelyhidi17} that $\mathcal{C}(M(n)//SO(n))$ can be identified with $\mathcal{C}_{SO(n)}(\mathbb{R}^n)$ via $f((k,x)) = f((id, kx))$ for all $(k,x) \in M(n)$. It was also shown in \cite[Section 5]{Szekelyhidi17} that $\mathcal{M}_c(M(n)//SO(n))$ can be identified with $\mathcal{M}_{SO(n)}(\mathbb{R}^n)$. For $f \in \mathcal{C}(\mathbb{R}^n)$, the projection $f^{\#}(x)$ is in $\mathcal{C}_{SO(n)}(\mathbb{R}^n)$ and is given by
\[ f^{\#}(x) = \int_{SO(n)} f(kx)\, dk. \]
Let $y \in \mathbb{R}^n$, recall that $(R_y\chi_E)(x) = \chi_E(x+y)$ is the right translate of $\chi_E$ by $y$. Since $\langle f, (R_y\chi_E)^{\#} \rangle = \langle f^{\#}, R_y\chi_E \rangle$ we see that for $y$ and $x$ in $\mathbb{R}^n$, 
\[ (R_y\chi_E)^{\#}(x) = \int_K \chi_E(kx + y)\, dk. \]
Let $(k,x) \in \mathcal{M}(n)$ and let $\widecheck{\chi}_{\widetilde{E}} \ast \chi_{\widetilde{ySO(n)}} \in M_c(M(n)//SO(n))$, where $y \in \mathbb{R}^n$.  Since $\widecheck{\chi}_{\widetilde{E}} \ast \chi_{\widetilde{ySO(n)}}$ is $SO(n)$-invariant we obtain
\[ \left(\widecheck{\chi}_{\widetilde{E}} \ast \chi_{\widetilde{ySO(n)}}\right)(k,x) = \left(\widecheck{\chi}_{\widetilde{E}} \ast \chi_{\widetilde{ySO(n)}}\right)(id,x) = \]
\[ \int_{\widetilde{SO(n)}} \widecheck{\chi}_{\widetilde{E}} (k^{-1}, k^{-1}(-y) + x)\, d\widetilde{SO(n)} = \int_{\widetilde{SO(n)}} \widecheck{\chi}_{\widetilde{E}} ((k^{-1},0)(id,kx -y))\, d\widetilde{SO(n)} = \]
\[ \int_{\widetilde{SO(n)}} \widecheck{\chi}_{\widetilde{E}} ((id, kx - y))\, d\widetilde{SO(n)} = \int_{SO(n)} \widecheck{\chi}_E (kx-y)\, dSO(n). \]

Hence, $\widecheck{\chi}_{\widetilde{E}} \ast \chi_{\widetilde{ySO(n)}}$ corresponds to $(R_{(-y)}\widecheck{\chi}_E)^{\#} \in M_{SO(n)}(\mathbb{R}^n)$. By Theorem \ref{mainresult} $E$ has the Pompeiu property if and only if $\bigcap_{y \in \mathbb{R}^n} Z((R_y\chi_E)^{\#}) = \emptyset$. 

The $SO(n)$-spherical functions in $\mathcal{C}_{SO(n)} (\mathbb{R}^n)$ are of the form
\[ \phi_{\lambda}(x) = \int_{S^{n-1}} e^{i\lambda (x \cdot w)}\, dw, \]
where $\lambda \in \mathbb{C}$ and $S^{n-1}$ is the $n$-sphere with normalized measure $dw$. For $\phi_{\lambda} \in \mathcal{S}$ and $\mu \in \mathcal{M}_{SO(n)} (\mathbb{R}^n)$, the transform $\Phi_{\phi_{\lambda}}$ from Section \ref{mainresultpaper} becomes 
\[ \Phi_{\phi_{\lambda}}(\mu) = \langle \phi_{\lambda}, \widecheck{\mu} \rangle = \int_{\mathbb{R}^n} \widecheck{\phi_{\lambda}}(x)\, d\mu(x). \]
We shall write $\Phi_{\lambda}$ for $\Phi_{\phi_{\lambda}}$. In this setting we can restate Corollary \ref{cormainresult} as:
\begin{Prop}\label{convequiv}
Let $E$ be a compact subset of $\mathbb{R}^n$ with positive Lebesgue measure. Recall the $\widetilde{E}$ is the lift of $E$ to $M(n)$. Then $E$ has the Pompeiu property if and only if $\phi_{\lambda} \ast \widecheck{\chi}_{\widetilde{E}} \neq 0$ for all $\lambda \in \mathbb{C}$. 
\end{Prop}
Thus, $E$ does not have the Pompeiu property if and only if there exists a $\lambda \in \mathbb{C}$ for which 
\begin{equation}
\phi_{\lambda} \ast \widecheck{\chi}_{\widetilde{E}} = 0. \label{eq:convkill}
\end{equation}
Recall that if $\mu \in \mathcal{M}_c(\mathbb{R}^n)$, then its Fourier-Laplace transform is given by 
\[ \mathcal{L}(\mu)(z) = \int_{\mathbb{R}^n} e^{-i(z \cdot x)}\, d\mu(x), \]
where $z \in \mathbb{C}^n$ and $z \cdot x = z_1 x_1 + z_2x_2 + \cdots + z_nx_n$ is the usual inner product. Set
\[ Z(\mathcal{L}(\mu)) = \{ z \in \mathbb{C}^n \mid \mathcal{L}(\mu) (z) = 0 \}, \]
where $\mu \in \mathcal{M}_c(\mathbb{R}^n)$.  For $\lambda \in \mathbb{C}$, set $S_{\mathbb{C}} (\lambda) = \{z \in \mathbb{C}^n \mid z_1^2 + z_2^2 + \cdots + z_n^2 = \lambda \}$.  The following was proved in \cite[Theorem 4.1]{BrownSchreiberTaylor73}:
\begin{Thm}\label{BST73}
Let $E$ be a compact subset of $\mathbb{R}^n$ with positive Lebesgue measure. Then $E$ does not have the Pompeiu property if and only if $S_{\mathcal{C}}(\lambda) \subseteq Z(\mathcal{L}(\chi_E))$ for some $\lambda \in \mathbb{C}.$
\end{Thm}
We will connect our results to the above theorem by showing that the $\lambda \in \mathbb{C}$ is the same for the theorem and (\ref{eq:convkill}). 

For $(k,x) \in M(n)$, we have that 
\[ \phi_{\lambda} \ast \widecheck{\chi}_{\widetilde{E}} (k,x) = \phi_{\lambda} \ast \widecheck{\chi}_{kE}(x) \]
for $\lambda \in \mathbb{C}$ because $\phi_{\lambda} \in \mathcal{M}_{SO(n)} (\mathbb{R}^n)$. Suppose $E$ does not have the Pompeiu property. Proposition \ref{convequiv} tells us that $\phi_{\lambda} \ast \widecheck{\chi}_{\widetilde{E}} = 0$ for some $\lambda \in \mathbb{C}$. Consequently, at the $\mathbb{R}^n$ level, $E$ does not have the Pompeiu property if and only if $\phi_{\lambda} \ast \widecheck{\chi}_{kE} = 0$ for all $k \in SO(n)$. Observe that when $E$ has positive Lebesgue measure, then (\ref{eq:convkill}) is not satisfied when $\lambda =0$ since $\phi_0 \ast \widecheck{\chi}(0)$ is the Lebesgue measure of $E$. Denote by $I$ the ideal in $\mathcal{M}_C (\mathbb{R}^n)$ generated by $\{ \chi_{kE} \mid k \in SO(n)\}$. It is immediate that $\phi_{\lambda} \in I^{\perp}$. Since $I$ is rotation and translation invariant, it satisfies spectral synthesis \cite[Theorem 3.1]{BrownSchreiberTaylor73}. This means that $e^{i(\vec{\lambda} \cdot x)} \in I^{\perp}$, where $\vec{\lambda} = (\lambda, 0, \dots, 0)$. Hence, $\vec{\lambda} \in Z(\mathcal{L}(\chi_{kE}))$ for each $k \in SO(n)$.  This implies $\vec{\lambda}(SO(n))$ is a real submanifold of $S_{\mathbb{C}}(\lambda)$ on which the analytic function $\mathcal{L}(\chi_E)$ vanishes due to the fact $\mathcal{L}(\chi_{kE})(z) = \mathcal{L}(\chi_E)(k^{-1}z)$ for all $k \in SO(n)$.  It now follows from \cite[Lemma 3.1]{Lebl24} that $\mathcal{L}(\chi_E)$ vanishes on $S_{\mathbb{C}} (\lambda)$.  We have just shown the following: 
\begin{Prop}\label{link}
Let $E$ be a compact subset of $\mathbb{R}^n$ with positive Lebesgue measure. Then there exists a $\lambda \in \mathbb{C}$ that satisfies $\phi_{\lambda} \ast \widecheck{\chi}_{\widetilde{E}} =0$ if and only if $S_{\mathbb{C}}(\lambda) \subseteq Z(\mathcal{L}(\chi_E))$.
\end{Prop}

 It was shown in \cite[Corollary 1.3]{MachadoRobins21} that for polytopes $E$ in $\mathbb{R}^n, S_{\mathbb{C}}(\lambda) \not\subseteq Z(\mathcal{L}(\chi_E))$ for all $\lambda \in \mathbb{C}\setminus \{0\}$, which established the result mentioned in the Introduction that polytopes in $\mathbb{R}^n$ have the Pompeiu property. 

If it is the case that $\chi_E \in \mathcal{M}_{SO(n)}(\mathbb{R}^n)$, then we can weaken the above conditions for $E$ to have the Pompeiu property. Specifically, 
\begin{Cor}\label{spherical}
Let $E$ be a compact subset of $\mathbb{R}^n$ with positive Lebesgue measure and assume $\chi_E \in \mathcal{M}_{SO(n)}(\mathbb{R}^n)$. Then $E$ has the Pompeiu property if and only if $\phi_{\lambda} \ast \chi_E \neq 0$ for all $\lambda \in \mathbb{C}$ if and only if $Z(\mathcal{L}(\chi_E)) = \emptyset$. 
\end{Cor}
\begin{proof}
This is essentially a restatement of Remark \ref{Eradial}. In the paragraphs above we saw that $E$ does not have Pompeiu property if and only if there exists a $\lambda \in \mathbb{C}$ that satisfies $\phi_{\lambda} \ast \chi_{kE} = 0$ for all $k \in SO(n)$. We only need to check that $\phi_{\lambda} \ast \chi_E =0$ since $\chi_E = \chi_{kE}$ for all $k \in SO(n)$. 
\end{proof}

\bibliographystyle{plain}
\bibliography{spectralpompeiuproblem}

\end{document}